\documentclass[10pt, a4paper, fleqn, final]{article}
\usepackage[utf8]{inputenc}
\usepackage[T1]{fontenc}
\usepackage[english]{babel}
\usepackage{csquotes}
\usepackage{amsmath, amsthm, amssymb, mathtools}
\usepackage[maxnames=4]{biblatex}
\usepackage{libertine}
\usepackage{dsfont}
\usepackage{hyperref, url}

\newtheorem{theo}{Theorem}
\newtheorem{coro}[theo]{Corollary}

\theoremstyle{definition}
\newtheorem{defi}[theo]{Definition}
\newtheorem{rema}[theo]{Remark}

\allowdisplaybreaks[4]

\newcommand{\eqspace}{\ensuremath{\mathrel{\phantom{=}}}}

\newcommand{\cdotcup}{\mathrel{\mathaccent\cdot\cup}}
\newcommand{\esg}[2]{#1\langle#2\rangle} 
\newcommand{\vig}[2]{#1[#2]} 
\DeclarePairedDelimiter\abs{\lvert}{\rvert}
\newcommand{\pcoef}[2]{[#1](#2)} 
\newcommand{\pdeg}[2]{\deg_{#1}(#2)} 

\title{Proving properties of \\ the edge elimination polynomial \\ using equivalent graph polynomials}
\author{Martin Trinks\thanks{trinks@hs-mittweida.de, Hochschule Mittweida, University of Applied Sciences, Faculty Mathematics / Sciences / Computer Science, Technikumplatz 17, 09648 Mittweida, Germany}}
\date{\today}

\begin{document}

\maketitle

{\center\setlength{\parindent}{0pt}
\begin{minipage}[c]{5.7cm}
The author receives the grant 080940498 from the European Social Fund (ESF) of the European Union (EU).
\end{minipage}
\hspace{0.4cm}
\begin{minipage}[c]{5.7cm}
\centering
\includegraphics[width=3cm]{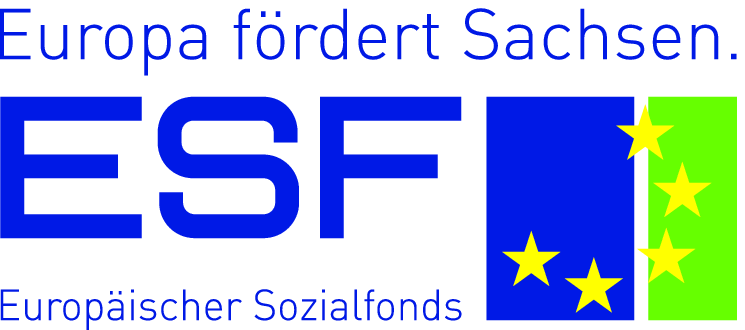}
\hspace{0.25cm}
\includegraphics[width=2cm]{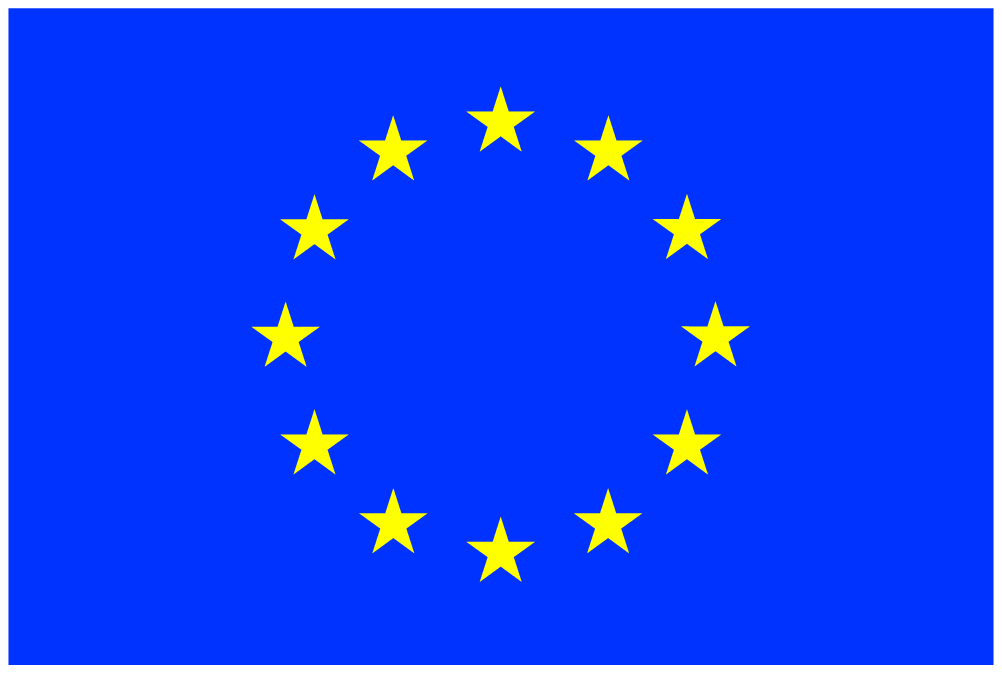}
\end{minipage}
}

\begin{abstract}
\citeauthor{averbouch2008} define the edge elimination polynomial of a graph by a recurrence relation with respect to the deletion, contraction and extraction of an edge. It generalizes some well-known graph polynomials such as the chromatic polynomial and the matching polynomial. By introducing two equivalent graph polynomials, one enumerating subgraphs and the other enumerating colorings, we show that the edge elimination polynomial of a simple graph is reconstructible from its polynomial deck and that it encodes the degree sequence of an arbitrary graph.
\end{abstract}


\section{Introduction}

In the literature there is a multitude of graph polynomials defined \cite{ellis2011b, noy2003}. Among them are several satisfying recurrence relations with respect to the deletion, contraction and extraction of an edge. \citeauthor{averbouch2008} \cite{averbouch2008, averbouch2010} determine the most general graph polynomial obeying such a recurrence relation, the \emph{edge elimination polynomial}.

Even though the edge elimination polynomial generalizes some well-known graph polynomials (as it generalizes their recurrence relations), only little is known about additional combinatorial information it encodes.

We introduce two graph polynomials and prove that both are equivalent to the edge elimination polynomial. The first, the \emph{subgraph counting polynomial}, counts subgraphs with respect to their number of vertices, edges and connected components. The second, the \emph{trivariate chromatic polynomial}, is the straightforward generalization of both the bad coloring polynomial \cite{welsh1993} and the bivariate chromatic polynomial \cite{dohmen2003}. Obviously, both graph polynomials provide a lot of combinatorial data of a graph. 

We prove that the subgraph counting polynomial of a simple graph with at least three vertices is reconstructible from its polynomial deck and that the trivariate chromatic polynomial of an arbitrary graph encodes the degree sequence. Furthermore, we give the relation 
to the \emph{subgraph component polynomial} \cite{tittmann2011}.

Other graph polynomials equivalent to the edge elimination polynomial are given by \citeauthor{white2011} \cite{white2011}, by \citeauthor{averbouch2011} \cite{averbouch2011}, and by the present author \cite{trinks2012}. 
Actually, there is one more equivalent graph polynomial in a publication several years ago, the \emph{subgraph enumerating polynomial} defined by \citeauthor{borzacchini1982b} \cite{borzacchini1982b}. The authors also state a recurrence relation \cite[Theorem 2]{borzacchini1982b} similar to the one used for the definition of the edge elimination polynomial.

For the sake of convenience we speak only about graphs and graph polynomials, but we also have hypergraphs and polynomials associated to them in mind. Straightforwardly generalizing the definitions and replacing ``graph'' by ``hypergraph'', all the results stating recurrence relations and the encoding of the degree sequence, that is all results except Theorem \ref{theo:rel_H_Q} and Theorem \ref{theo:recon_H}, stay valid (including the proofs). Results concerning the edge elimination polynomial of hypergraphs are also given by \citeauthor{white2011} \cite{white2011} and the present author \cite[Section 9]{trinks2012}.

\subsection{Graph theory}

A \emph{hypergraph} $G = (V, E)$ is an ordered pair of a set of vertices, the vertex set $V$, and a multiset of edges, the edge set $E$, such that each edge is a non-empty subset of the vertex set, i.e. $e \subseteq V$ for all $e \in E$.

A \emph{graph} $G = (V, E)$ is a hypergraph, such that each edge is a one- or two-element subset of the vertex set, i.e. $e \in \binom{V}{1} \cup \binom{V}{2}$ for all $e \in E$.  A \emph{simple graph} $G = (V, E)$ is a graph, such that each edge is a two-element subset of the vertex set and the edge set is a set, i.e. $E \subseteq \binom{V}{2}$.

For a graph $G = (V, E)$ with a vertex $v \in V$ and an edge $e \in E$, we say that $v$ and $e$ are \emph{incident} to each other, if $v \in e$. The \emph{degree} of $v$ is the number of edges incident to $v$, and we denote the \emph{number of vertices with degree $i$} by $d(G, i)$.  By $V(G)$, $E(G)$ and $k(G)$ we refer to the vertex set, the edge set and the \emph{number of connected components}, respectively.

Further assuming a vertex subset $W \subseteq V$ and an edge subset $F \subseteq E$, we consider the following different types of subgraphs: A graph $H = (W, F)$ is a \emph{subgraph} of $G$, denoted by $H \subseteq G$. The graph $\esg{G}{F} = (V, F)$ is a \emph{spanning subgraph} of $G$ and the graph $\vig{G}{W} = (W, \{e \in E \mid e \subseteq W\})$ is an \emph{induced subgraph} of $G$.

We use the following graph operations:
\begin{itemize}
\item $-v$: \emph{deletion} of the vertex $v$, i.e. vertex $v$ and its incident edges are removed,
\item $-e$: \emph{deletion} of the edge $e$, i.e. edge $e$ is removed,
\item $/e$: \emph{contraction} of the edge $e$, i.e. edge $e$ is removed and its incident vertices are merged (parallel edges and loops may occur),
\item $\dagger e$: \emph{extraction} of the edge $e$, i.e. vertices incident to edge $e$ are deleted
.
\end{itemize}

The graphs arising are denoted by $G_{-v}$, $G_{-e}$, $G_{/e}$ and $G_{\dagger e}$, respectively. For two graphs $G^1$ and $G^2$, $G^1 \cdotcup G^2$ denotes the \emph{disjoint union} of $G^1$ and $G^2$, i.e. the union of disjoint copies of both graphs. $K_n$ is the \emph{complete graph} on $n$ vertices.

Two graphs are \emph{isomorphic} to each other, if one arises from the other by a relabeling of the vertices (in the vertex set and in the edges). An \emph{(invariant) graph polynomial} is a function, which maps isomorphic graphs to the same polynomial (of some polynomial ring over a set $X$ of commuting variables, for example $\mathds{R}[X]$). A graph polynomial is \emph{equivalent} to another graph polynomial, if both can be calculated from each other. (We only apply variable substitution and multiplication with a constant factor.) For a (graph) polynomial $P$, a monomial $m$ and a variable $z$, we denote by $\pcoef{m}{P}$ the coefficient of $m$ in $P$ and by $\pdeg{z}{P}$ the degree of $z$ in $P$.

For all other notations refer to \cite{diestel2010}.

\subsection{The edge elimination polynomial}

The edge elimination polynomial \cite{averbouch2008, averbouch2010} is defined by a recurrence relation.

\begin{defi}[Equation (13) in \cite{averbouch2008}]
\label{defi:eep}
Let $G=(V, E), G^1, G^2$ be graphs and $e \in E$ an edge of $G$. The \emph{edge elimination polynomial} $\xi(G, x, y, z)$ is defined as
\begin{align}
& \xi(G, x, y, z) = \xi(G_{-e}, x, y, z) + y \cdot \xi(G_{/e}, x, y, z) + z \cdot \xi(G_{\dagger e}, x, y, z), \label{eq:defi_eep_1} \\
& \xi(G^1 \cdotcup G^2, x, y, z) = \xi(G^1, x, y, z) \cdot \xi(G^2, x, y, z) \label{eq:defi_eep_2}, \\
& \xi(K_1, x, y, z) = x. \label{eq:defi_eep_3}
\end{align}
\end{defi}

Averbouch, Godlin and Makowsky \cite[Theorem 3]{averbouch2008} prove that the edge elimination polynomial is the most general graph polynomial, which satisfies a linear recurrence relation with respect to the three edge operations and is multiplicative in components, that is the graph polynomial of the disjoint union of two graphs equals the product of the graph polynomials of the two graphs, as given in Equation \ref{eq:defi_eep_2}.

\subsection{Potts model, bad coloring polynomial and \texorpdfstring{\\}{} bivariate chromatic polynomial}

There is a multitude of graph polynomials generalized by the edge elimination polynomial \cite[Subsection 1.1]{averbouch2008}. We introduce some of them that we use in the following: 
the Potts model, the bad coloring polynomial and the bivariate chromatic polynomial. 

The \emph{Potts model} $Z(G, x, y)$ \cite{sokal2005} 
of a graph $G = (V, E)$ is the generating function for the number of connected components and edges in the spanning subgraphs:
\begin{align}
Z(G, x, y) = \sum_{A \subseteq E}{x^{k(\esg{G}{A})} y^{\abs{A}}}.
\end{align}

It is also known as the ``bivariate version of the multivariate Tutte polynomial'' \cite[Equation (2.1)]{sokal2005} and as ``bivariate partition function of the $q$-state Potts model'' \cite[Section 1.1]{averbouch2010}. The Potts model $Z(G) = Z(G, x, y)$ satisfies \cite[Equation (4.16) and (4.1)]{sokal2005}
\begin{align}
& Z(G) = Z(G_{-e}) + y \cdot Z(G_{/e}), \label{eq:Z_rec1} \\
& Z(G^1 \cdotcup G^2) = Z(G^1) \cdot Z(G^2), \label{eq:Z_rec2} \\
& Z(K_1) = x. \label{eq:Z_rec3}
\end{align}

The \emph{bad coloring polynomial} $\tilde{\chi}(G, x, z)$ \cite{welsh1993} of a graph $G = (V, E)$ is the generating function of the number of (vertex) colorings with (at most) $x$ colors with respect to the number of monochromatic edges, which are edges whose incident vertices are all mapped to the same color:
\begin{align}
\tilde{\chi}(G, x, z) 
&= \sum_{\phi \colon V \rightarrow \{1, \ldots, x\}}{\prod_{\substack{e \in E \\ \exists c \forall v \in e \colon \phi(v) = c}}{z}}.
\end{align}
It is also known as ``monochrome polynomial'' \cite{welsh1993} and ``coboundary polynomial'' \cite{goodall2006}. The bad coloring polynomial $\tilde{\chi}(G) = \tilde{\chi}(G, x, z)$ satisfies \cite[Equation (4.4.3)]{welsh1993}
\begin{align}
& \tilde{\chi}(G) = \tilde{\chi}(G_{-e}) + (z - 1) \cdot \tilde{\chi}(G_{/e}), \\
& \tilde{\chi}(G^1 \cdotcup G^2) = \tilde{\chi}(G^1) \cdot \tilde{\chi}(G^2), \\
& \tilde{\chi}(K_1) = x.
\end{align}

The \emph{bivariate chromatic polynomial} $P(G, x, y)$ \cite{dohmen2003} of a graph $G = (V, E)$ is the number of (vertex) colorings with (at most) $x$ colors, such that the vertices of monochromatic edges are not colored with one of the first $y$ colors:
\begin{align}
P(G, x, y) 
&= \sum_{\phi \colon V \rightarrow \{1, \ldots, x\}}{\prod_{\substack{e \in E \\ \exists c \leq y \forall v \in e \colon \phi(v) = c}}{0}}.
\end{align}
The bivariate chromatic polynomial $P(G) = P(G, x, y)$ satisfies \cite[Proposition 1]{averbouch2008}
\begin{align}
& P(G) = P(G_{-e}) - P(G_{/e}) + (x-y) \cdot P(G_{\dagger e}), \\
& P(G^1 \cdotcup G^2) = P(G^1) \cdot P(G^2), \\
& P(K_1) = x.
\end{align}

From their recurrence relations it follows that all three graph polynomials, the Potts model, the bad coloring polynomial and the bivariate chromatic polynomial, are specializations of the edge elimination polynomial:
\begin{alignat}{2}
& Z(G, x, y) &&= \xi(G, x, y, 0), \\
& \tilde{\chi}(G, x, z) &&= \xi(G, x, z-1, 0), \\
& P(G, x, y) &&= \xi(G, x, -1, x-y).
\end{alignat}
The first two are equivalent to each other and are strongly related to the Tutte polynomial \cite{tutte1954, ellis2011}; all three are generalizations of the chromatic polynomial \cite{birkhoff1912, dong2005}.

\section{The subgraph counting polynomial}

The subgraph counting polynomial generalizes the Potts model by summing over all subgraphs instead only spanning subgraphs. We show that this polynomial is equivalent to the edge elimination polynomial.

\begin{defi}
\label{defi:h}
Let $G = (V, E)$ be a graph. The \emph{subgraph counting polynomial} $H(G, v, x, y)$ is defined as
\begin{align}
H(G, v, x, y)
&= \sum_{H = (W, F) \subseteq G}{v^{\abs{W}} x^{k(H)} y^{\abs{F}}}.
\end{align}
\end{defi}

The subgraph counting polynomial can be stated as
\begin{align}
H(G, v, x, y)
&= \sum_{W \subseteq V}{v^{\abs{W}} \cdot Z(\vig{G}{W}, x, y)}, \label{eq:H_exp_1}
\end{align}
and it is therefore related to the Potts model $Z(G, x, y)$ by 
\begin{align}
Z(G, x, y) = \pcoef{v^{\abs{V}}}{H(G, v, x, y)}. \label{eq:H_Z}
\end{align}

\begin{theo}
\label{theo:scp_rec}
Let $G = (V, E), G^1, G^2$ be graphs and $e \in E$ an edge of $G$. The subgraph counting polynomial $H(G) = H(G, v, x, y)$ satisfies
\begin{align}
& H(G) = H(G_{-e}) + v^{\abs{e} - 1} y \cdot H(G_{/e}) - v^{\abs{e} - 1} y \cdot H(G_{\dagger e}), \label{eq:theo_scp_rec_1} \\
& H(G^1 \cdotcup G^2) = H(G^1) \cdot H(G^2) \label{eq:theo_scp_rec_2}, \\
& H(K_1) = 1 + vx. \label{eq:theo_scp_rec_3}
\end{align}
\end{theo}

\begin{proof}
The second equality holds as the subgraphs in different components can be chosen independently from each other and the third one holds by definition. Therefore, it only remains to show the first equality.

Let $[W', F']$ be the subgraph counting polynomial counting only those subgraphs including exactly the vertex set $W'$ from the vertices incident to $e$ and exactly the edge set $F'$ from the edge set consisting only of the edge $e$, i.e. those subgraphs $H = (W, F)$ with $W \cap e = W'$ and $F \cap \{e\} = F'$. We determine the subgraphs counted by the graphs arising from deletion, contraction and extraction of $e$.

The subgraphs of $G_{-e}$ are the subgraphs of $G$ not including $e$ but a (possible empty) subset of the incident vertices of $e$, i.e.
\begin{alignat*}{2}
& H(G_{-e}) &&= \sum_{U \subseteq e}{[U, \emptyset]}. \\
\intertext{The subgraphs of $G_{/e}$ can be include the vertex $w$, to which $e$ is contracted, or not. In the first case, the subgraphs (including $w$) can be mapped to subgraphs including $e$ and its incident vertices (contracting $e$ keeps the connection properties), but dived by $v^{\abs{e}-1} y$ (all vertices except one and $e$ are not counted). In the second case, the subgraphs (not including $w$) are the subgraphs of $G$ not including $e$ or one of its incident vertices. Combining both we get}
& H(G_{/e}) &&= \frac{[e, \{e\}]}{v^{\abs{e}-1} y} + [\emptyset, \emptyset]. \\
\intertext{The subgraphs of $G_{\dagger e}$ are the subgraphs of $G$ not including $e$ or one of its incident vertices, i.e.}
& H(G_{\dagger} e) &&= [\emptyset, \emptyset].
\end{alignat*}
Consequently, the recurrence relation equals the sum of the distinct cases:
\begin{align*}
H(G_{-e}) + v^{\abs{e} - 1} y \cdot H(G_{/e}) - v^{\abs{e} - 1} y \cdot H(G_{\dagger e}) 
&= \sum_{U \subseteq e}{[U, \emptyset]} + [e, \{e\}] \\
&= H(G). \qedhere
\end{align*}
\end{proof}

\begin{coro}
\label{coro:rel_scp_eep}
Let $G = (V, E)$ be a graph. The subgraph counting polynomial $H(G, v, x, y)$ and the edge elimination polynomial  $\xi(G, x, y, z)$ are equivalent graph polynomials related by
\begin{alignat}{2}
& H(G, v, x, y) &&= v^{\abs{V}} \cdot \xi(G, \frac{1 + vx}{v}, y, - \frac{y}{v}), \label{eq:coro_rel_scp_eep} \\
& \xi(G, x, y, z) &&= (x-y)^{\abs{V}} \cdot H(G, \frac{1}{x - y}, y, \frac{z}{x-y}). \label{eq:coro_rel_eep_scp}
\end{alignat}
\end{coro}

\begin{proof}
Let $\bar{H} = \bar{H}(G, v, x, y) = v^{-\abs{V}}\cdot H(G, v, x, y)$. We show that $\bar{H}(G, v, x, y) = \xi(G, \frac{1 + vx}{v}, y, - \frac{y}{v})$, from this the first equality follows directly and the second equality follows by algebraic transformations. The graph polynomials $\bar{H}(G, v, x, y)$ and $\xi(G, \frac{1 + vx}{v}, y, - \frac{y}{v})$ have the same initial value ($\frac{1 + vx}{v}$) and are both multiplicative in components. Hence, it only remains to show that $\bar{H}(G, v, x, y)$ satisfies the same recurrence relation as $\xi(G, \frac{1 + vx}{v}, y, - \frac{y}{v})$. This follows from the recurrence relation of $H(G) = H(G, v, x, y)$ given in the theorem above by
\begin{align*}
\bar{H}(G)
&= v^{- \abs{V}} \cdot H(G) \\
&= v^{- \abs{V}} \cdot [H(G_{-e}) + v^{\abs{e} - 1} y \cdot H(G_{/e}) - v^{\abs{e} - 1} y \cdot H(G_{\dagger e})] \\
&= v^{- \abs{V}} \cdot H(G_{-e}) + v^{- \abs{V} + \abs{e} - 1} y \cdot H(G_{/e}) \\
& \eqspace - v^{- \abs{V} + \abs{e} - 1} y \cdot H(G_{\dagger e}) \\
&= v^{- \abs{V(G_{-e})}} \cdot H(G_{-e}) + yv^{- \abs{V(G_{/e})}} \cdot H(G_{/e}) \\
& \eqspace - \frac{y}{v} v^{- \abs{V(G_{\dagger e})}} \cdot H(G_{\dagger e}) \\
&= \bar{H}(G_{-e}) + y \cdot \bar{H}(G) - \frac{y}{v} \cdot \bar{H}(G). \qedhere 
\end{align*}
\end{proof}

\section{The trivariate chromatic polynomial}

The bad coloring polynomial generalizes the chromatic polynomial by counting the number of monochromatic edges (also known as ``bad edges''). The bivariate chromatic polynomial generalizes the chromatic polynomial by allowing a subset of the colors to appear in monochromatic edges. We show that the graph polynomial combining both generalizations, the trivariate chromatic polynomial\footnote{The present author has introduced this graph polynomial under the name ``bivariate bad coloring polynomial'' in several talks, first time at a conference at the Zhejiiang Normal University (Jinhua, China) in 2010. Because of the conflict between a ``bivariate'' polynomial and using three variables, the name is changed into the same used by \citeauthor{white2011} \cite{white2011} for an almost similar graph polynomial.}, is equivalent to the edge elimination polynomial. 

\begin{defi}
\label{defi:tcp_def}
Let $G = (V, E)$ be a graph. The \emph{trivariate chromatic polynomial} $\tilde{P}(G, x, y, z)$ is defined (for $x, y \in \mathds{N}$) as
\begin{align}
\tilde{P}(G, x, y, z) 
&= \sum_{\phi \colon V \rightarrow \{1, \ldots, x\}}{\prod_{\substack{e \in E \\ \exists c \leq y \forall v \in e \colon \phi(v) = c}}{z}}.
\end{align}
\end{defi}

The above definition equals the ``trivariate chromatic polynomial'' $P(G, p, q, t)$ of \citeauthor{white2011} \cite[Section 6]{white2011} except a change of the first two variables, that is $P(G, p, q, t) = \tilde{P}(G, q, p, t)$. Furthermore, it is utilized by \citeauthor{garijo2011} \cite[Theorem 34]{garijo2011} to prove that the edge elimination polynomial can be stated as counting graph homomorphisms, more precisely $\hom(G, K^1_{q-p} + K^y_p) = \tilde{P}(G, q, p, y)$.

The trivariate chromatic polynomial is the generating function for the number of monochromatic edges, whose incident vertices are mapped to the same color $c \leq y$, in the (vertex) colorings with (at most) $x$ colors.

In fact, we first can select a set of vertices which we color (independently) by one of the $x-y$ colors $y+1, \ldots, x$ and then color the remaining vertices with one of the $y$ colors $1, \ldots, y$ enumerating the number of monochromatic edges. Thus, the trivariate chromatic polynomial can be stated as
\begin{align}
\tilde{P}(G, x, y, z)
&= \sum_{W \subseteq V}{(x-y)^{\abs{W}} \cdot \tilde{\chi}(G_{-W}, y, z)}, \label{eq:tP_exp_1}
\end{align}
and it is therefore related to the bad coloring polynomial $\tilde{\chi}(G, y, z)$ and the bivariate chromatic polynomial $P(G, x, y)$ by
\begin{alignat}{2}
& \tilde{\chi}(G, x, z) &&= \pcoef{v^{0}}{\tilde{P}(G, v+x, x, z)} = \tilde{P}(G, x, x, z), \\
& P(G, x, y) &&= \tilde{P}(G, x, y, 0).
\end{alignat}

\begin{theo}
\label{theo:tcp_rec}
Let $G = (V, E)$, $G^1$, $G^2$ be graphs and $e \in E$ an edge of $G$. The trivariate chromatic polynomial $\tilde{P}(G) = \tilde{P}(G, x, y, z)$ satisfies
\begin{align}
& \tilde{P}(G) = \tilde{P}(G_{-e}) + (z-1) \cdot \tilde{P}(G_{/e}) + (1-z) (x-y) \cdot \tilde{P}(G_{\dagger e}), \label{eq:theo_tcp_rec_1} \\
& \tilde{P}(G^1 \cdotcup G^2) = \tilde{P}(G^1) \cdot \tilde{P}(G^2), \label{eq:theo_tcp_rec_2} \\
& \tilde{P}(K_1) = x. \label{eq:theo_tcp_rec_3}
\end{align}
\end{theo}

\begin{proof}
We only prove the first equality, the other two follow from the definition and are in full analogy to the chromatic polynomial and the mentioned generalizations.

For the coloring of the vertices incident to the edge $e$ there are the following three distinct cases:
\begin{enumerate}
\item $e$ is not monochromatic, i.e. not all vertices of $e$ are mapped to the same color $c$: $\nexists c \colon \forall v \in e \colon \phi(v) = c$,
\item $e$ is ``bad monochromatic'', i.e. all vertices of $e$ are mapped to the same color $c \leq y$: $\exists c \leq y \colon \forall v \in e \colon \phi(v) = c$,
\item $e$ is ``good monochromatic'', i.e. all vertices of $e$ are mapped to the same color $c > y$: $\exists c > y \colon \forall v \in e \colon \phi(v) = c$.
\end{enumerate}

Let $p_1$, $p_2$ and $p_3$ be the trivariate chromatic polynomial of $G$ enumerating exactly those colorings of $G$ corresponding to the first, second and third case, respectively. Obviously, $\tilde{P}(G) = p_1 + p_2 + p_3$.

Each coloring of the vertices of $G_{-e}$ corresponds to a coloring of the vertices of $G$, where the number of bad monochromatic edges is counted correctly, except in the second case, that is the vertices incident to $e$ (in $G$) are colored by the same color $c \leq y$. Then $e$ is not counted as bad monochromatic (as it does not appear in the graph):
\begin{align*}
\tilde{P}(G_{-e}) = p_1 + \frac{p_2}{z} + p_3.
\end{align*}

Each coloring of the vertices of $G_{/e}$ corresponds to a coloring of the vertices of $G$, where all vertices incident to $e$ are mapped to the color $c$, to which the vertex arising through the contraction of $e$ is mapped, which covers the second and third case. But again, in the second case the edge is not counted as bad monochromatic:
\begin{align*}
\tilde{P}(G_{/e}) = \frac{p_2}{z} + p_3.
\end{align*}

Each coloring of the vertices of $G_{\dagger e}$ corresponds to a colorings of the vertices of $G$ excluding the vertices incident to $e$. If we assume that the vertices of $e$, as in the third case, are all colored by the same color $c > y$, then there are $x - y$ for them:
\begin{align*}
\tilde{P}(G_{\dagger e}) = \frac{p_3}{x-y}.
\end{align*}

The statement follows by
\begin{align*}
\tilde{P}(G_{-e}) + (z-1) \cdot \tilde{P}(G_{/e}) + (1-z) (x-y) \cdot \tilde{P}(G_{\dagger e}) &= p_1 + p_2 + p_3 \\
&= \tilde{P}(G). \qedhere
\end{align*}
\end{proof}

\begin{coro}
\label{coro:rel_tcp_eep}
Let $G = (V, E)$ be a graph. The trivariate chromatic polynomial $\tilde{P}(G, x, y, z)$ and the edge elimination polynomial $\xi(G, x, y, z)$ are equivalent graph polynomials related by
\begin{alignat}{2}
& \tilde{P}(G, x, y, z) &&= \xi(G, x, z-1, (1-z)(x-y)), \label{eq:theo_rel_tcp_eep} \\
& \xi(G, x, y, z) &&= \tilde{P}(G, x, x + \frac{z}{y}, y+1). \label{eq:theo_rel_eep_tcp}
\end{alignat}
\end{coro}

\begin{proof}
The first equality follows directly from the theorem above and the second one by algebraic transformations.
\end{proof}

\section{Properties}

In this section we prove some properties valid for the edge elimination polynomial and graph polynomials equivalent to it, by using the combinatorial interpretations of the newly defined graph polynomials.

\subsection{Relation to the subgraph component polynomial}

The \emph{subgraph component polynomial} $Q(G, x, y)$ \cite{tittmann2011, averbouch2010b} of a graph $G = (V, E)$ is the generating function for the number of connected components in the subgraphs induced by the vertex subsets:
\begin{align}
Q(G, v, x) = \sum_{W \subseteq V}{v^{\abs{W}} x^{k(\vig{G}{W})}}.
\end{align}

It is known that the subgraph component polynomial of the line graph $L(G)$ can be derived from the edge elimination polynomial of $G$ 
\cite[Theorem 23]{tittmann2011}. We show that for forests the subgraph counting polynomial and the subgraph component polynomial are equivalent to each other.

\begin{theo}
\label{theo:rel_H_Q}
Let $F = (V, E)$ be a forest. The subgraph counting polynomial $H(F, v, x, y)$ and the subgraph component polynomial $Q(F, v, x)$ are related by
\begin{alignat}{2}
& H(F, v, x, y) &&= Q(F, v (x + y), \frac{x}{x+y}), \label{eq:theo_rel_H_Q} \\
& Q(F, v, x) &&= H(F, v, x, 1-x). \label{eq:theo_rel_Q_H}
\end{alignat}
\end{theo}

\begin{proof}
We use the expansion of the subgraph counting polynomial as a sum of Potts models of induced subgraphs given in Equation \eqref{eq:H_exp_1}. The Potts model of a forest depends only on the number of vertices and the number of connected components, hence we have
\begin{align*}
H(F, v, x, y)
&= \sum_{W \subseteq V}{v^{\abs{W}} \cdot Z(\vig{F}{W}, x, y)} \\
&= \sum_{W \subseteq V}{v^{\abs{W}} x^{k(\vig{F}{W})} (x+y)^{\abs{W} - k(\vig{F}{W})}}.
\end{align*} 
Then Equation \eqref{eq:theo_rel_H_Q} follows by
\begin{align*}
Q(F, v (x + y), \frac{x}{x+y})
&= \sum_{W \subseteq V}{(v (x+y))^{\abs{W}} \left( \frac{x}{x+y} \right)^{k(\vig{F}{W})}} \\
&= \sum_{W \subseteq V}{v^{\abs{W}} x^{k(\vig{F}{W})} (x+y)^{\abs{W} - k(\vig{F}{W})}} \\
&= H(F, v, x, y),
\end{align*}
and Equation \eqref{eq:theo_rel_Q_H} by
\begin{align*}
H(F, v, x, 1-x)
&= \sum_{W \subseteq V}{v^{\abs{W}} x^{k(\vig{F}{W})} (1)^{\abs{W} - k(\vig{F}{W})}} \\
&= \sum_{W \subseteq V}{v^{\abs{W}} x^{k(\vig{F}{W})}} \\
& = Q(F, v, x). \qedhere
\end{align*}
\end{proof}

\begin{rema}
For general graphs neither the subgraph component polynomial can be obtained from the subgraph counting polynomial nor the other way around. This can be observed as follows: The graphs $G^1$ and $G^2$ in Figure \ref{fig:graphs_rel_scp_scomp} \cite[$G_3$ and $G_4$ in Figure 2]{trinks2012} have the same subgraph counting polynomial (because both have the same edge elimination polynomial), but different subgraph component polynomials, notice for example the coefficient of $v^6 x^1$. The graphs $G^3$ and $G^4$ in Figure \ref{fig:graphs_rel_scp_scomp} \cite[$G_7$ and $G_8$ in Figure 1]{averbouch2010b} have the same subgraph component polynomial, but different subgraph counting polynomials (because both have different chromatic polynomials). For non-simple graphs this follows already from the fact that the subgraph component polynomial does not consider parallel edges and loops, which the subgraph counting polynomial does.
\end{rema}

\begin{figure}
\begin{center}
	\includegraphics{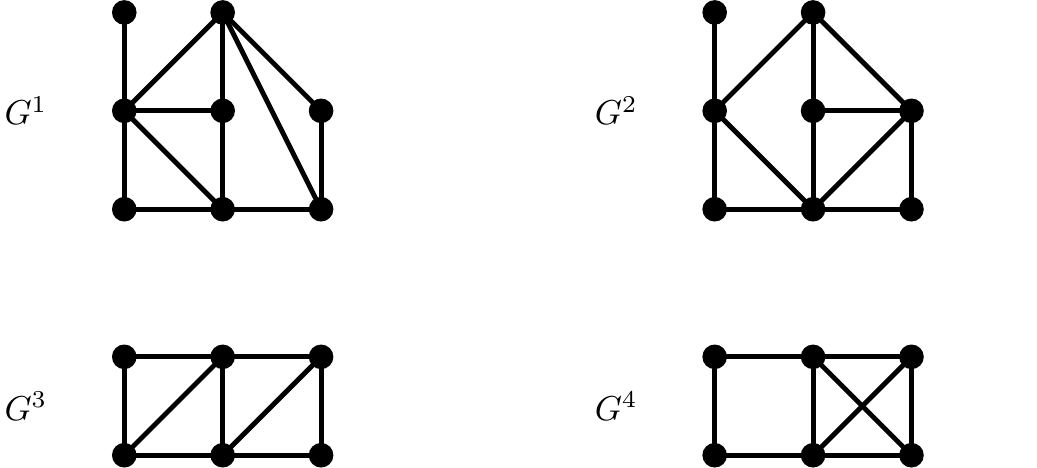}
\end{center}
\caption{$G^1$ and $G^2$ are graphs with same subgraph counting polynomial, but different subgraph component polynomials, $G^3$ and $G^4$ are graphs with same subgraph component polynomial, but different subgraph counting polynomials.}
\label{fig:graphs_rel_scp_scomp}
\end{figure}

\subsection{Polynomial reconstructibility}

The \emph{reconstruction conjecture} of \citeauthor{kelly1957} \cite{kelly1957} and \citeauthor{ulam1960} \cite{ulam1960} states that every simple graph $G = (V, E)$ with at least three vertices can be reconstructed from the isomorphism classes of its \emph{deck} $\mathcal{D}(G)$, which is the set of vertex-deleted subgraphs, i.e. $\mathcal{D}(G) = \{G_{-v} \mid v \in V\}$. (In the following we assume only simple graph with at least three vertices and do not mention this restriction.) \citeauthor{kotek2009} \cite[Theorem 2.5]{kotek2009} shoes that the edge elimination polynomial of a simple graph is reconstructible from the isomorphism classes of the deck of the graph.

For a graph polynomial $P(G)$, reconstruction can be ``restricted'' as follows: Can the value of $P(G)$ for a graph be reconstructed from its \emph{polynomial deck} $\mathcal{D}_P(G)$, which is the multiset of graph polynomials of the graphs in the deck, i.e. $\mathcal{D}_P(G) = \{P(G_{-v}) \mid v \in V\}$. 

The probably first affirmative statement in this direction is given by \citeauthor{tutte1967} for the \emph{rank polynomial} \cite{tutte1967, tutte1979}. Additionally knowing the number of vertices and of connected components, it is equivalent to the Potts model, and consequently this is also reconstructible from its polynomial deck.

\begin{theo}
\label{theo:recon_H}
Let $G = (V, E)$ be a simple graph with at least three vertices. The subgraph counting polynomial $H(G, v, x, y)$ of $G$ is reconstructible from the polynomial deck $\mathcal{D}_H(G)$.
\end{theo}

\begin{proof}
We use the vertex subset expansion of the subgraph counting polynomial given in Equation \eqref{eq:H_exp_1}:
\begin{align*}
H(G, v, x, y)
&= \sum_{W \subseteq V}{v^{\abs{W}} \cdot Z(\vig{G}{W}, x, y)}.
\end{align*}
Consequently, in the sum of the polynomials in the polynomial deck each summand of $H(G, v, x, y)$ including $v^{i}$ arises $(\abs{V} - i)$-times. (Each such summand is part of the Potts model of a induced subgraph with $i$ vertices and each of these graphs is a induced subgraph of $\abs{V} - i$ graphs in the deck --- analogous to Kelly's Lemma \cite[Lemma]{kelly1957}.)

Hence, only the summands including $v^{\abs{V}}$ are missing, which correspond to the Potts model of $G$. From the subgraph counting polynomial of a graph we can calculate its Potts model by Equation \eqref{eq:H_Z}. Therefore, we can calculate from the polynomial deck for the subgraph counting polynomial the polynomial deck for the Potts model, and from this the Potts model of $G$ can be determined via the result of Tutte.
\end{proof}

\subsection{Degree sequence}

A graph invariant of typical interest with respect to the question whether it is encoded in a graph polynomial or not, is the degree sequence of a graph. This is the non-increasing sequence of degrees of the vertices of a graph and can be determined from the trivariate chromatic polynomial via the number of vertices of degree $i$.

\begin{theo}
\label{theo:degseq_H}
Let $G = (V, E)$ be a graph. The number of vertices with degree $i$ for $i \in \{0, \ldots, \abs{E}\}$, $d(G, i)$, is encoded in the trivariate chromatic polynomial $\tilde{P}(G, x, y, z)$:
\begin{align}
d(G, i) = \pcoef{v^{1} z^{\abs{E} - i}}{\tilde{P}(G, v+1, 1, z)},
\end{align}
where $\abs{E} = \pdeg{z}{\tilde{P}(G, v+1, 1, z)}$.
\end{theo}

Informal we do the following: We consider the number of bad monochromatic edges arising by a coloring using $v$ ``arbitrary'' colors and $1$ ``proper'' color. Each term including $v^{1}$ corresponds to a coloring where exactly one vertex is colored by one of the ``arbitrary'' colors and all other vertices are colored by the same ``proper'' color. Hence, all edges except the edges incident to the one ``arbitrary colored'' vertex are bad monochromatic, and their number is counted in the variable $z$.

\begin{proof}
We start with the expansion of the trivariate chromatic polynomial given in Equation \eqref{eq:tP_exp_1} and apply the edge subset expansion of the bad coloring polynomial \cite[Section 9.6.2]{ellis2011}:
\begin{align*}
\tilde{P}(G, x, y, z)
&= \sum_{W \subseteq V}{(x-y)^{\abs{W}} \cdot \tilde{\chi}(G_{-W}, y, z)} \\
&= \sum_{W \subseteq V}{(x-y)^{\abs{W}} \sum_{A \subseteq E(G_{-W})}{y^{k(G_{-W})} (z-1)^{\abs{A}}}}.
\end{align*}
For $\tilde{P}(G, v+1, 1, z)$ it follows
\begin{align*}
\tilde{P}(G, v+1, 1, z)
&= \sum_{W \subseteq V}{v^{\abs{W}} \sum_{A \subseteq E(G_{-W})}{1^{k(G_{-W})} (z-1)^{\abs{A}}}} \\
&= \sum_{W \subseteq V}{v^{\abs{W}} z^{\abs{E(G_{-W})}}}.
\end{align*}
Consequently, the coefficients in front of the terms $v^{1} z^{\abs{E}-i}$ count the number of vertices, whose deletion removes $i$ edges and hence the number of vertices with degree $i$, $d(G, i)$. 
\end{proof}

\section{Conclusion}

We have presented two graph polynomials that are defined in different frameworks but equivalent to the edge elimination polynomial. Thereby we have related a definition using recurrence relations to one counting subgraphs and to another counting colorings. Furthermore, we have applied these definitions to prove properties valid for the whole class of ``edge elimination polynomials''.

We think the results are a good motivation for the searching and investigation of equivalent graph polynomials. As evidence for this, compare 
the proof for the fact that the degree sequence of forests is encoded in the edge elimination polynomial \cite[Theorem 36]{trinks2012} with the almost trivial proof of the more general result that the degree sequence of arbitrary graphs is encoded in the trivariate chromatic polynomial (Theorem \ref{theo:degseq_H}). 

\printbibliography

\end{document}